\documentclass[11pt]{article}
\usepackage{amssymb}
\usepackage{latexsym,amsmath,amssymb,amsfonts,epsfig,graphicx,cite,psfrag,dsfont}
\usepackage{eepic,color,colordvi,amscd}
\usepackage{mathrsfs}
\usepackage{tikz}
\usepackage{subfigure}
\usepackage[english]{babel}
\usepackage[center]{caption2}
\usepackage{amsthm}
\usepackage{multirow}
\usepackage[usenames,dvipsnames]{pstricks}
\usepackage{epsfig}
\usepackage{pst-grad} 
\usepackage{pst-plot} 
\usepackage[space]{grffile} 
\usepackage{etoolbox} 
\usepackage{enumerate}
\usepackage{amsthm,amscd}
\usepackage[all]{xy}
\usepackage{color}
\usepackage{tikz}

\newtheorem{theo}{Theorem}[section]

\newtheorem{lem}[theo]{Lemma}

\newtheorem{claim}{Claim}
\newtheorem*{cl}{Claim}

\begin{document}

\title{A note on cycle lengths in graphs of chromatic number five and six}
\author{Qingyi Huo\thanks{School of Mathematical Sciences, University of Science and Technology of China, Hefei, Anhui 230026, China. Email: qyhuo@mail.ustc.edu.cn.
Partially supported by NSFC grant 11622110, National Key Research and Development Project SQ2020YFA070080, and Anhui Initiative in Quantum Information Technologies grant AHY150200.}
}
\date{}

\maketitle

\begin{abstract}
In this note, we prove that every non-complete $(k+1)$-critical graph contains cycles of all lengths modulo $k$, where $k=4,5$.
Together with a result in \cite{GHM20}, this completely gives an affirmative answer to the question of Moore and West on graphs of given chromatic number.
\end{abstract}

\section{Introduction}

The problem of deciding whether a given graph contains cycles of all lengths modulo a positive integer $k$ shows up in many literatures (see \cite{B77,CY,D10,Erd76,Fan02,GHLM,GHM20,LM,SV17,Th83,Th88,V00}).
Recently, Moore and West \cite[Question 2]{MW} asked whether every $(k+1)$-critical non-complete graph has a cycle of length $2$ modulo $k$. Here, a graph is $k$-{\it critical} if it has chromatic number $k$ but deleting any edge will decrease the chromatic number.
Very recently, Gao, Huo and Ma \cite{GHM20} partially answered this question by showing the following theorem.

\begin{theo}[\cite{GHM20} Theorem 1.4]\label{old}
For $k\geq 6$, every non-complete $(k+1)$-critical graph contains cycles of all lengths modulo $k$.
\end{theo}

However, methods in \cite{GHM20} do not work for $k<6$.
In this note, we give a new method and prove that the conclusion of Theorem \ref{old} also holds for $k=4,5$.

\begin{theo}\label{maintheorem_chromatic}
For $k=4,5$, every non-complete $(k+1)$-critical graph contains cycles of all lengths modulo $k$.
\end{theo}

Thus, combined with the Theorems \ref{old} and \ref{maintheorem_chromatic}, we completely give an affirmative answer to the question of Moore and West.
See \cite{GHM20,MW} for the history and further references about cycle lengths in graphs of given chromatic number.

The rest of the paper is organized as follows.
In Section~\ref{definition}, we introduce the notation.
In Section~\ref{SEClemma}, we give a key lemma.
In Section \ref{SECfive}, we consider graphs of chromatic number five and prove Theorem \ref{maintheorem_chromatic} for the case $k=4$.
In Section \ref{SECsix}, we consider graphs of chromatic number six and prove Theorem \ref{maintheorem_chromatic} for the case $k=5$.

\section{Notation}\label{definition}

All graphs considered are finite, undirected, and simple.
Let $G$ be a graph and let $H$ be a subgraph of a graph $G$.
We say that $H$ and a vertex $v\in V(G)-V(H)$ are {\it adjacent} in $G$ if $v$ is adjacent in $G$ to some vertex in $V(H)$.
Let $N_G(H):=\bigcup_{v\in V(H)}N_G(v)-V(H)$ and $N_G[H]:=N_G(H)\cup V(H)$.
For a subset $S$ of $V(G)$, $G[S]$ denotes the subgraph induced by $S$ in $G$, and $G-S$ denotes the subgraph $G[V(G)-S]$.
A vertex is a {\it leaf} in $G$ if it has degree one in $G$.
We say that a path $P$ is {\it internally disjoint} from $H$ if no vertex of $P$ other than its endpoints is in $V(H)$.
For two vertex-disjoint subgraphs $H, H'$ of $G$, let $N_{H}(H')$ be the set of vertices in $H$ which is adjacent to some vertex in $H'$.

A cycle or a path is said to be \textit{odd} (resp.~\textit{even})
if its length is odd (resp.~even).
Given a cycle $C$ and an orientation of $C$,
for two vertices $x$ and $y$ in $C$,
we denote by $C[x,y]$ the path on $C$ from $x$ to $y$ in
the direction,
including $x$ and $y$.
Let $C[x,y) := C[x,y] - y$,
$C(x,y] := C[x,y] - x$,
and
$C(x,y) := C[x,y] - \{x,y\}$.
We use the similar notation to a path $P$.

Let $u$ and $v$ be vertices of a graph.
If there are three internally disjoint paths between $u$ and $v$,
then we call such a graph as \emph{theta graph}.
Note that any theta graph contains an even cycle.

A vertex $v$ of a graph $G$ is a {\em cut-vertex} of $G$ if $G-v$ contains more components than $G$.
A {\em block} $B$ in $G$ is a maximal connected subgraph of $G$ such that there exists no cut-vertex of $B$.
So a block is an isolated vertex, an edge or a $2$-connected graph.
An {\em end-block} in $G$ is a block in $G$ containing at most one cut-vertex of $G$.
If $D$ is an end-block of $G$ and a vertex $x$ is the only cut-vertex of $G$ with $x \in V(D)$, then we say that $D$ is an {\em end-block with cut-vertex $x$}.

Let $T$ be a tree, and fix a vertex $r$ as its root.
Let $v$ be a vertex of $T$.
The $\it{parent}$ of $v$ is the vertex adjacent to $v$ on the path from $v$ to $r$.
An $\it{ascendant}$ of $v$ is any vertex which is either the parent of $v$ or is recursively the ascendant of the parent of $v$.
A $\it{child}$ of $v$ is a vertex of which $v$ is the parent.
A $\it{descendant}$ of $v$ is any vertex which is either the child of $v$ or is recursively the descendant of any of the children of $v$.
Let $Y$ be a subset of $V(T)$. We say a vertex $x$ is the $\it{descendant}$ of $Y$ if $x$ is the descendant of some vertex in $Y$.
Let $a,b$ be two vertices of $T$.
Denote $T_{a,b}$ the unique path between $a$ and $b$ in $T$.

\section{Key lemma}\label{SEClemma}

Let $G$ be a $2$-connected graph and let $C$ and $D$ be two cycles in $G$.
We say that $(C,D)$ is {\it an opposite pair} in $G$, if $C$ is odd and $D$ is even satisfying that $C$ and $D$ are edge-disjoint and share at most one common vertex.

\begin{lem}\label{mainlemma}
Let $G$ be a $2$-connected graph of minimum degree at least $4$.
Let $(C,D)$ be an opposite pair in $G$.
Then $G$ contains cycles of all lengths modulo $4$.
\end{lem}

\begin{proof}
Suppose to the contrary that $G$ does not contain cycles of all lengths modulo $4$.
Since $G$ is $2$-connected and $|V(C)\cap V(D)|\leq 1$,
there exist two vertex disjoint paths $P,Q$ between $C$ and $D$ satisfying $(V(C)\cap V(D))-V(Q)=\emptyset$.\footnote{We remarked that (i) if $V(C)\cap V(D)=\emptyset$, then $P$ and $Q$ are vertex disjoint, (ii) if $C$ and $D$ share one common vertex, then $V(Q)=V(C)\cap V(D)$.}
We take such an opposite pair $(C,D)$, paths $P$ and $Q$ as the following manner:
\begin{enumerate}[(1)]
\item $|E(P)|$ is as large as possibly,
\item $|E(Q)|$ is as large as possible subject to (1).
\end{enumerate}

Let $p$ and $q$ be the endpoints of $P$ and $Q$ in $D$, respectively.

\begin{claim}\label{maximize}
Every even cycle in the block of $G-(V(C\cup P\cup Q)-\{p,q\})$ including $D$ contains both $p$ and $q$.
In particular, every theta graph in the block includes both $p$ and $q$.
\end{claim}

\begin{proof}[Proof of Claim \ref{maximize}]
Let $H$ be the block of $G-(V(C\cup P\cup Q)-\{p,q\})$ including $D$.
Let $D'$ be an even cycle in $H$ other than $D$.
Suppose that $p\notin V(D')$.
Since $H$ is $2$-connected, there are two vertex disjoint paths $L_1,L_2$ from $\{p,q\}$ to $D'$ in $H$.
We may assume that $L_1$ links $p$ and $D'$.
Note that $L_1$ has length at least $1$ and $(C,D')$ is an opposite pair.
Then $P\cup L_1$ and $Q\cup L_2$ are two internally disjoint paths between $C$ and $D'$ such that $P\cup L_1$ is longer than $P$, a contradiction.
Therefore, $p\in V(D')$.

Suppose that $q\notin V(D')$.
Since $H$ is $2$-connected, there is a path $L_3$ from $q$ to $D'$ internally disjoint from $V(D')$ in $H$.
Note that $L_3$ has length at least $1$ and $(C,D')$ is an opposite pair.
Then $P$ and $Q\cup L_3$ are two internally disjoint paths between $C$ and $D'$ such that $Q\cup L_3$  is longer than $Q$, a contradiction.
Therefore, $q\in V(D')$.
Since every theta graph contains an even cycle, every theta graph in $H$ includes both $p$ and $q$.
This completes the proof of Claim \ref{maximize}.
\end{proof}

Since $D$ is an even cycle, we partition $V(D)$ into the sets $A$ and $B$ alternatively along $D$.
By symmetry between $A$ and $B$, we may assume that $p\in A$.

\begin{claim}\label{nopath_B_to_CPQ}
For any $b \in B-\{q\}$,
there is no path from $b$ to $C \cup P \cup Q - \{p,q\}$ internally disjoint from $C \cup D \cup P \cup Q$.
\end{claim}

\begin{proof}[Proof of Claim \ref{nopath_B_to_CPQ}]
Suppose to the contrary that there is a path $R$ from $b$ to $x\in V(C\cup P\cup Q) - \{p,q\}$ internally disjoint from $C \cup D \cup P \cup Q$.
By symmetry, we may assume that $b\in D(p,q)$.

Assume that $|E(D)|\equiv 0$ modulo $4$.
As $C$ is an odd cycle, there is an even path $X_1$ and an odd path $Y_1$ between $p$ and $q$ in $C \cup P \cup Q$.
If $q \in B$, then both $|E(D[p,q])|$ and $|E(D[q,p])|$ are odd,
and furthermore, since their sum is $0$ modulo $4$, they differ by $2$ modulo $4$.
Then $X_1 \cup D[p,q], X_1 \cup D[q,p], Y_1 \cup D[p,q]$ and $Y_1 \cup D[q,p]$ are $4$ cycles of different lengths modulo $4$, a contradiction.
Therefore, we have that $q \in A$.

\begin{itemize}
\item Suppose that $x\in V(P)-\{p\}$.
Since $C$ is an odd cycle,
there is an even path $X_2$ and an odd path $Y_2$ between $b$ and $q$ in $C \cup P \cup Q \cup R$.
However, since both $|E(D[b,q])|$ and $|E(D[q,b])|$ are odd and differ by $2$ modulo $4$,
$X_2 \cup D[b,q], X_2 \cup D[q,b], Y_2 \cup D[b,q]$ and $Y_2 \cup D[q,b]$ are $4$ cycles of different lengths modulo $4$, a contradiction.
Thus, $x$ is not contained in $V(P)-\{p\}$.
\item Suppose that $x\in V(C\cup Q) - (V(P)\cup\{q\})$.
Then there is an even path $X_3$ and an odd path $Y_3$ between $b$ and $p$ in $C \cup P \cup Q \cup R$.
However, since both $|E(D[b,p])|$ and $|E(D[p,b])|$ are odd and differ by $2$ modulo $4$,
$X_3 \cup D[b,p], X_3 \cup D[p,b], Y_3 \cup D[b,p]$ and $Y_3 \cup D[p,b]$ are $4$ cycles of different lengths modulo $4$, a contradiction.
Thus, $x$ is not contained in $V(C\cup Q) - (V(P)\cup\{q\})$.
\end{itemize}

Therefore, $|E(D)|\equiv 2$ modulo $4$.
As $C$ is an odd cycle, there is an even path $X_4$ and an odd path $Y_4$ between $p$ and $q$ in $C \cup P \cup Q$.
If $q \in A$, then both $|E(D[p,q])|$ and $|E(D[q,p])|$ are even,
and furthermore, since their sum is $2$ modulo $4$, they differ by $2$ modulo $4$.
Then $X_4 \cup D[p,q], X_4 \cup D[q,p], Y_4 \cup D[p,q]$ and $Y_4 \cup D[q,p]$ are $4$ cycles of different lengths modulo $4$, a contradiction.
Therefore, we have that $q \in B$.

\begin{itemize}
\item Suppose that $x\in V(C\cup P)-(V(Q)\cup\{p\})$.
Since $C$ is an odd cycle,
there is an even path $X_5$ between $b$ and $q$ and an odd path $Y_5$ between $b$ and $q$ in $C \cup P \cup Q \cup R$.
However, since both $|E(D[b,q])|$ and $|E(D[q,b])|$ are odd and differ by $2$ modulo $4$,
$X_5 \cup D[b,q], X_5 \cup D[q,b], Y_5 \cup D[b,q]$ and $Y_5 \cup D[q,b]$ are $4$ cycles of different lengths modulo $4$, a contradiction.
Thus, $x$ is not contained in $V(C\cup P)-(V(Q)\cup\{p\})$.

\item Suppose that $x\in V(Q) -\{q\}$.
Since $G$ is $2$-connected and $G$ has minimum degree at least $4$, there exists a path $T$ from $b$ to $y\in V(C\cup D\cup P\cup Q\cup R)-\{b\}$ internally disjoint from $C \cup D \cup P \cup Q\cup R$.
Based on previous analysis, we have that $y\in V(Q\cup D\cup R)-\{b\}$.

\begin{itemize}
\item If $y\in V(R\cup Q\cup D(b,p))-\{b\}$, then $D[b,p) \cup R \cup T\cup Q$ contains a theta graph.
It follows that there is an even $D_1$ cycle in $G-(C\cup P-Q)$.
Note that $(C,D_1)$ is an opposite pair in $G$.
It is easy to see that there are two internally disjoint paths $P'$ and $Q'$ between $C$ and $D'$ satisfying that $P'$ contains $P$ and is longer than $P$ and $Q'\subseteq Q\cup D(b,q]$, a contradiction.
Thus, $y$ is not contained in $V(R\cup Q\cup D(b,p))-\{b\}$.

\item Suppose that $y\in V(D[p,b))$.
Since $T\cup D[y,b]$ does not contain $q$ and $D[b,q]\cup R\cup Q[x,q]$ does not contain $p$, by the choice of opposite pairs, we have that $T\cup D[y,b]$ and $D[b,q]\cup R\cup Q[x,q]$ are both odd cycles.
Since $C$ is an odd cycle, there is an odd path $X'$ and an even path $Y'$ between $y$ and $x$ in $C \cup P \cup Q\cup D[p,y]$.
Note that $X'$ and $Y'$ differ by $1$ modulo $4$, $T$ and $D[y,b]$ differ by $1$ modulo $4$ and $D[b,q]\cup Q[x,q]$ and $R$ differ by $1$ modulo $4$.
Then the set $\{L_1\cup L_2\cup L_3|L_1\in\{X',Y'\},L_2\in\{T,D[y,b]\},L_3\in\{D[b,q]\cup Q[x,q],R\}\}$ contains cycles of all lengths modulo $4$, a contradiction.
Thus, $y$ is not contained in $V(D[p,b))$.
\end{itemize}

\end{itemize}

This completes the proof of Claim \ref{nopath_B_to_CPQ}.
\end{proof}

Let $z$ be a vertex in $B-\{q\}$.
By symmetry, we may assume that $z \in V(D(p,q))$.
Since $z$ has degree at least $4$ in $G$ and $G$ is $2$-connected,
there is a path $Z$ from $z$ to $C \cup D \cup P \cup Q - \{z\}$ internally disjoint from $C \cup D \cup P \cup Q$.
By Claim \ref{nopath_B_to_CPQ}, the endpoint of $Z$ other than $z$ is contained in $D - \{z\}$.
Let $r$ be the endpoint of $Z$ other than $z$.
Since $z$ has degree at least $4$ in $G$ and $G$ is $2$-connected,
there is a path $S$ from $z$ to $s\in V(C \cup D \cup P \cup Q \cup Z)- \{z\}$ internally disjoint from $C \cup D \cup P \cup Q \cup Z$.
By Claim \ref{nopath_B_to_CPQ}, $s$ is contained in $V(D\cup Z) - \{z\}$.

\begin{itemize}

\item Suppose that $s\in V(Z)-\{z\}$.

\begin{itemize}
\item If $r\in V(D(z,p))$, then $D[z,r]\cup Z\cup S$ is a theta graph not containing $p$, contradicting Claim \ref{maximize}.
\item If $r\in V(D[p,z))$, then $D[r,z]\cup Z\cup S$ is a theta graph not containing $q$, contradicting Claim \ref{maximize}.
\end{itemize}

Thus, $s$ is not contained in $V(Z)-\{z\}$.
\item Suppose that $s\in D-\{z,r\}$.
By symmetry between $r$ and $s$, we may assume that $s \in V(D(r,z))$.

\begin{itemize}
\item If $r\in V(D(q,z))$, then $D[r,z]\cup Z\cup S$ is a theta graph not containing $q$, contradicting Claim \ref{maximize}.
\item If $r\in V(D(z,q])$ and $s\in V(D(r,p))$, then $D[z,s]\cup Z\cup S$ is a theta graph not containing $p$, contradicting Claim \ref{maximize}.
\item Therefore $r\in V(D(z,q])$ and $s\in V(D[p,z))$.
Since $S\cup D[s,z]$ does not contain $q$ and $D[z,r]\cup Z$ does not contain $p$, by Claim \ref{maximize}, we have that $S\cup D[s,z]$ and $D[z,r]\cup Z$ are both odd cycles.
Since $C$ is an odd cycle, there is an odd path $X''$ and an even path $Y''$ between $s$ and $r$ in $C \cup P \cup Q\cup D[p,s]\cup D[r,q]$.
Note that $X''$ and $Y''$ differ by $1$ modulo $4$, $S$ and $D[s,z]$ differ by $1$ modulo $4$ and $D[z,r]$ and $Z$ differ by $1$ modulo $4$.
Then the set $\{L_1\cup L_2\cup L_3|L_1\in\{X'',Y''\},L_2\in\{S,D[s,z]\},L_3\in\{D[z,r],Z\}\}$ contains cycles of all lengths modulo $4$, a contradiction.
\end{itemize}

\end{itemize}

This completes the proof of Lemma \ref{mainlemma}.
\end{proof}

\section{Graphs of chromatic number five}\label{SECfive}

In this section, we prove the following theorem on $2$-connected graphs of minimum degree at least four, from which Theorem \ref{maintheorem_chromatic} can be inferred as a corollary for the case $k=4$.

\begin{theo}\label{maintheorem}
Every $2$-connected non-bipartite graph of minimum degree at least $4$ contains cycles of all lengths modulo $4$, except that it is the complete graph of five vertices.
\end{theo}

\begin{proof}
Let $G$ be a $2$-connected non-bipartite graph of minimum degree at least $4$.
Assume that $G$ is not a $K_5$ and does not contain cycles of all lengths modulo $4$.
Let $C:=v_0v_1\ldots v_{2\ell}v_0$ be an odd cycle in $G$ such that $|V(C)|$ is minimum, where the indices are taken under the additive group $\mathbb{Z}_{2\ell+1}$.
Note that $C$ is induced.
Let $H:=G-V(C)$.
By Lemma \ref{mainlemma}, there is no opposite pairs in $G$, hence $H$ does not contain an even cycle.
It follows that every block of $H$ is either an odd cycle, an edge or an isolated vertex.

\begin{cl}
$G$ does not contain a triangle.
\end{cl}

\begin{proof}[Proof of Claim]
Suppose that $G$ contains a triangle.
Then $C$ is a triangle.
Let $H_1$ be a component of $H$.
Since $G$ has minimum degree at least $4$, $H_1$ has at least two vertices.
Suppose that $H_1$ contains an odd cycle $C_1$.
\begin{itemize}
\item If $H_1$ is not $2$-connected, then there exists an end-block $B_1$ of $H_1$ with cut-vertex $b_1$ such that $(V(B_1)-\{b_1\})\cap C_1=\emptyset$.
As $B_1$ is either an odd cycle or an edge, there exists $w\in V(B_1)-\{b_1\}$ such that $w$ has at least two neighbors on $C$.
Since $C$ is an odd cycle, $G[C\cup \{w\}]$ contains an even cycle $D_1$.
Then $C_1$ and $D_1$ form an opposite pair in $G$, a contradiction.
\item Therefore, $H_1$ is $2$-connected, that is $H_1$ is an induced odd cycle, we denote $H_1:=u_0u_1\ldots u_{2h}u_0$, where the indices are taken under the additive group $\mathbb{Z}_{2h+1}$.
Since $G$ has minimum degree at least $4$, $u_0$ and $u_2$ have at least two neighbors on $C$.
Without loss of generality, we may assume that $u_0$ is adjacent to $v_0$ and $v_1$ and $u_2$ is adjacent to $v_0$.
Then $C,u_0v_0v_2v_1u_0,u_0u_1u_2v_0v_1u_0$ and $u_0u_1u_2v_0v_2v_1u_0$ are cycles of lengths $3,4,5$ and $6$, respectively, a contradiction.
\end{itemize}

Therefore every component of $H$ does not contain an odd cycle, that is, every component of $H$ is a tree.
\begin{itemize}

\item If $|V(H_1)|=2$, then $G[C\cup H_1]$ is a $K_5$.
Suppose that there is another component $H_2\neq H_1$ of $H$.
Since $G$ is $2$-connected, there are two disjoint path $L_1$ and $L_2$ from $H_2$ to $C$ internally disjoint from $C$ in $G[H_2\cup C]$.
Without loss of generality, we may assume that $V(L_i)\cap V(C)=\{v_i\}$ for $i=1,2$.
Concatenating $L_1$, $L_2$ and a path in $H_2$, there exists a path $L$ from $v_1$ to $v_2$ internally disjoint from $C$ in $G[H_2\cup C]$. As there are paths of lengths $1,2,3$ and $4$ from $v_1$ to $v_2$ in $G[H_1\cup C]$, we could easily obtain $4$ cycles of consecutive lengths, a contradiction.
Therefore, $H=H_1$.
It follows that $G=G[C\cup H_1]$, a contradiction.

\item Therefore $|V(H_1)|\geq3$.
For any two leaves $x,y$ of $H_1$, let $T$ be the fixed path between $x$ and $y$ in $H_1$.
Since $G$ has minimum degree at least $4$, $x$ and $y$ have at least two neighbors on $C$.
Without loss of generality, we may assume that $x$ is adjacent to $v_0$ and $v_1$ and $y$ is adjacent to $v_0$.
If $T$ is even, then $C$ and $v_0yTxv_0$ form an opposite pair, a contradiction.
Therefore $T$ is odd.
Suppose that there exist three leaves $x,y$ and $z$ in $H_1$. Let $T_{x,y},T_{y,z}$ and $T_{z,x}$ be the fixed paths between $x$ and $y$, $y$ and $z$ and $z$ and $x$ in $H_1$, respectively.
Note that all of them are odd.
However, there sum is even, a contradiction.
Therefore, $H_1$ is a path.
Let $H_1:=z_0z_1z_2\ldots z_n$ for some $n\geq 2$.
Since $G$ has minimum degree at least $4$, $z_0$ is adjacent to all vertices of $C$ and $z_2$ is adjacent to at least $2$ vertices of $C$. Without loss of generality, we may assume that $z_2$ is adjacent to $v_0$ and $v_1$.
Then $C,z_0v_0v_2v_1z_0,z_0z_1z_2v_0v_1z_0$ and $z_0z_1z_2v_0v_2v_1z_0$ are cycles of lengths $3,4,5$ and $6$, respectively, a contradiction.
\end{itemize}
This completes the proof of Claim.
\end{proof}

By Claim, $G$ does not contain a triangle.
Suppose that there is a vertex $u$ of degree at most one in $H$.
Since $G$ has minimum degree at least $4$, $u$ has at least three neighbors on $C$.
Since $C$ is odd, there exist two distinct neighbors $v_i,v_j$ of $u$ on $C$ such that the odd path between $v_i$ and $v_j$ on $C$ has no internal vertices which are the neighbors of $u$ in $G$ .
Let $Q_o,Q_e$ be the odd and even paths between $v_i$ and $v_j$ in $C$ respectively.
Let $C':=uv_i\cup Q_o\cup v_ju$.
Note that $C'$ is an odd cycle.
By the choice of $C$, we have that $|E(C')|\geq|E(C)|$.
This fores that $|E(Q_e)|=2$ and $u$ is adjacent to all vertices of $V(Q_e)$.
It follows that there is a triangle in $G$, a contradiction.
Therefore, $H$ has minimum degree at least $2$.

Suppose that $H$ has more than one component.
Let $W_1$ and $W_2$ be two components of $H$.
Since every vertex in $W_1$ has degree at least $2$, we have that $W_1$ contains an odd cycle $C_2$.
Since $G$ is $2$-connected and $C$ is an odd cycle, there is an even cycle $D_2$ in $G[V(C)\cup W_2]$.
Thus, $C_2$ and $D_2$ form an opposite pair, a contradiction.
Therefore, $H$ is connected.

Note that $H$ has minimum degree at least $2$ and every block of $H$ is either an odd cycle, an edge or an isolated vertex.
There is a vertex $t$ of $H$ which has at least two neighbors on $C$.
Since $C$ is odd, there exist two distinct neighbors $v_i,v_j$ of $t$ on $C$ such that the odd path between $v_i$ and $v_j$ on $C$ has no internal vertices which are the neighbors of $t$ in $G$ .
Let $Q'_o,Q'_e$ be the odd and even paths between $v_i$ and $v_j$ in $C$ respectively.
Let $C'':=tv_i\cup Q'_o\cup v_jt$.
Note that $C''$ is an odd cycle.
By the choice of $C$, we have that $|E(C'')|\geq|E(C)|$.
This fores that $|E(Q'_e)|=2$.
Without loss of generality, we may assume that $i=j+2$.
Let $s$ be the neighbor of $v_{j+\ell+1}$ in $H$.
Since $H$ is connected, there is a path $L$ between $t$ and $s$ in $H$.
Then $C[v_{j+2},v_{j+\ell+1}]\cup v_{j+\ell+1}s\cup L\cup tv_{j+2},\ C[v_{j+\ell+1},v_j]\cup v_jt\cup L\cup sv_{j+\ell+1},\ C[v_j,v_{j+\ell+1}]\cup v_{j+\ell+1}s\cup L\cup tv_{j},\ C[v_{j+\ell+1},v_{j+2}]\cup v_{j+2}t\cup L\cup sv_{j+\ell+1}$ are 4 cycles of consecutive lengths, a contradiction.
This completes the proof of Theorem \ref{maintheorem}.
\end{proof}

We remark that Theorem \ref{maintheorem} is best possible by the following examples.
For any positive integer $t$, let $P_t:=v_0v_1\ldots v_{2t+1}$ and $Q_t:=u_0u_1\ldots u_{2t+1}$ be two vertex disjoint paths.
Let $H_t$ be the graph obtained from $P_t\cup Q_t$ by adding edges in
$\{v_{2i}u_{2i+1}, u_{2i}v_{2i+1}, u_0v_0, u_{2t+1}v_{2t+1}|i=0,1,\ldots,t\}$.
We see that $H_t$ is a $2$-connected non-bipartite graph of minimum degree $3$ without cycles of length $1$ modulo $4$.

\begin{figure}[h]
\centering{}
\begin{tikzpicture}[thick,scale=2]
\node (u0) at (0, 0){$u_0$};
\node (u1) at (1, 0){$u_1$};
\node (v0) at (0, 1){$v_0$};
\node (v1) at (1, 1){$v_1$};
\node (ui) at (2, 0){$u_{2i}$};
\node (ui1) at (3, 0){$u_{2i+1}$};
\node (vi) at (2, 1){$v_{2i}$};
\node (vi1) at (3, 1){$v_{2i+1}$};
\node (ut) at (4, 0){$u_{2t}$};
\node (ut1) at (5, 0){$u_{2t+1}$};
\node (vt) at (4, 1){$v_{2t}$};
\node (vt1) at (5, 1){$v_{2t+1}$};

\draw (u0) -- (u1);
\draw (u0) -- (v0);
\draw (v0) -- (u1);
\draw (v0) -- (v1);
\draw (u0) -- (v1);

\draw[dashed] (v1) -- (vi);
\draw[dashed] (u1) -- (ui);

\draw (vi) -- (vi1);
\draw (ui) -- (ui1);
\draw (vi) -- (ui1);
\draw (ui) -- (vi1);

\draw[dashed] (vi1) -- (vt);
\draw[dashed] (ui1) -- (ut);

\draw (ut) -- (ut1);
\draw (ut1) -- (vt1);
\draw (vt) -- (ut1);
\draw (vt) -- (vt1);
\draw (ut) -- (vt1);
\end{tikzpicture}
\caption{Graphs without cycles of length $1$ modulo $4$}
\end{figure}
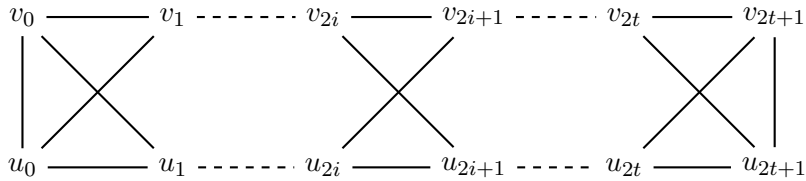

\section{Graphs of chromatic number six}\label{SECsix}

In this section, we consider graphs of chromatic number six and prove Theorem \ref{maintheorem_chromatic} for the case $k=5$.
We need the following theorems in \cite{GHM20}.

\begin{theo}[\cite{GHM20} Theorem 3.2]\label{A-B pathmodify}
Let $G$ be a connected graph of minimum degree at least three and $(A,B)$ be a non-trivial partition of $V(G)$.
For any cycle $C$ in $G$, there exist $A$-$B$ paths of every length less than $|V(C)|$ in $G$, unless $G$ is bipartite with the bipartition $(A,B)$.
\end{theo}

\begin{theo}[\cite{GHM20} Theorem 4.1]\label{longcycle}
Let $k\geq 3$ be an integer. Let $G$ be a $2$-connected graph of minimum degree at least $k$.
If $G$ is $K_3$-free, then $G$ contains a cycle of length at least $2k+2$, except that $G=K_{k,n}$ for some $n\geq k$.
\end{theo}

\begin{theo}[\cite{GHM20} Theorem 5.2]\label{containK3}
Let $k\geq 2$ be an integer.
Every $2$-connected graph $G$ of minimum degree at least $k$ containing a triangle $K_3$ contains $k$ cycles of consecutive lengths, except that $G=K_{k+1}$.
\end{theo}

\begin{theo}\label{six}
Every graph of chromatic number six contains cycles all lengths modulo five.
\end{theo}

\begin{proof}
It suffices to consider $6$-critical graphs $G$.
Suppose that $G$ does not contain five cycles of all lengths modulo five.
It is well-known that $G$ is a $2$-connected graph of minimum degree at least five.
By Theorem \ref{containK3}, we may assume that $G$ is $K_3$-free.
Fix a vertex $r$ and let $T$ be the breadth first search tree in $G$ with root $r$.
Let $L_0=\{r\}$ and $L_i$ be the set of vertices of $T$ at distance $i$ from its root $r$.

\begin{lem}\label{chromaticatmost3}
Every component of $G[L_i]$ has chromatic number at most 3, for all $i\geq0$ .
\end{lem}

\begin{proof}
Suppose to the contrary that there exists a component $D$ of $G[L_t]$ which has chromatic number at least $4$ for some $t$.
Let $H$ be a $4$-critical subgraph of $D$.
It is clear that $H$ is a $2$-connected non-bipartite graph of minimum degree at least $3$.
By Theorem \ref{longcycle}, $H$ contains a cycle of length at least $8$.
Let $T'$ be the minimal subtree of $T$ whose set of leaves is precisely $V(H)$, and let $r'$ be the root of $T'$.
Let $h$ denote the distance between $r'$ and vertices in $H$ in $T'$.
Since $G$ is $K_3$-free, $h\geq 2$.
By the minimality of $T'$, $r'$ has at least two children in $T'$.
Let $x$ be one of its children.
Let $A$ be the set of vertices in $H$ which are the descendants of $x$ in $T'$ and let $B=V(H)-A$.
Then both $A,B$ are nonempty and for any $a\in A$ and $b\in B$, $T_{a,b}$ has the same length $2h$.
By Theorem \ref{A-B pathmodify}, there are $7$ subpaths of $H$ from a vertex of $A$ to a vertex of $B$ of length $1,2,\ldots, 7$, respectively.
It follows that $G$ contains $7$ cycles of consecutive lengths, a contradiction.
This completes the proof of Lemma \ref{chromaticatmost3}.
\end{proof}

For a connected graph $D$, a vertex in $D$ is called {\it good} if it is not contained in the minimal connected subgraph of $D$ which contains all $2$-connected blocks of $D$, and {\it bad} otherwise.

\begin{lem}\label{keyclam}
Let $H_1$ be a non-bipartite component of $G[L_i]$ and $H_2$ be a non-bipartite component of $G[L_{i+1}]$ for some $i\geq 1$. If $N_{H_1}(H_2)\neq\emptyset$,
then every vertex in $N_{H_1}(H_2)$ is a good vertex of $H_1$.
\end{lem}

\begin{proof}
Suppose that there exists a bad vertex $v$ of $H_1$ which has a neighbor in $H_2$.
Let $T'$ be the minimal subtree of $T$ whose set of leaves is precisely $V(H_1)$,
and let $r'$ be the root of $T'$.
Let $h$ denote the distance between $r'$ and vertices in $H_1$ in $T'$.
Since $G$ is $K_3$-free, $h \geq 2$.
By the minimality of $T'$, $r'$ has at least two children in $T'$.
Let $(X,Y)$ be a non-trivial partition of all children of $r'$ in $T'$.
Let $A$ be the set of vertices in $H_1$ which are the descendants of $X$ in $T'$ and
let $B$ be the set of vertices in $H_1$ which are the descendants of $Y$ in $T'$.
Note that $(A,B)$ is a non-trivial partition of $V(H_1)$.
Note that every vertex in $B$ is the descendants of $Y$ in $T'$.
Let $A'$ be the set of vertices in $L_i-A$ which are the descendants of $X$ in $T$.
Let $B'$ be the set of vertices in $L_i-B$ which are the descendants of $Y$ in $T$.
Let $M:=L_i-(A\cup A'\cup B\cup B')$.
Note that $A,A',B,B'$ and $M$ form a partition of $L_i$.
Note that every vertex of $H_2$ has a neighbor in $L_i$.

Suppose that there exists a vertex $m\in V(H_2)$ which has a neighbor $m'$ in $M$.
Recall that $H_1$ is non-bipartite and $K_3$-free.
There exists a path $z_1z_2z_3z_4z_5$ of length $4$ in $H_1$ with $z_1=v$.
It is easy to see that $T_{z_i,m}$ contains $r'$ for $i\in[5]$, so they have the same length.
Let $P$ be a fixed path from $u$ to $m$ in $H_2$.
Then $P\cup uz_1z_2\ldots z_i \cup T_{z_i,m'}\cup m'm$, for $i\in[5]$ are $5$ cycles of consecutive lengths in $G$, a contradiction.
Therefore $N_M(H_2)=\emptyset$, that is every vertex in $H_2$ has a neighbor in $A\cup A'\cup B\cup B'$.
For a vertex in $V(H_2)$, we call it {\it type-$A$} if it has a neighbor in $A\cup A'$ and it {\it type-$B$} if it has a neighbor in $B\cup B'$.
\footnote{We remark that a vertex can be both type-$A$ and type-$B$.}

Let $C=v_0v_1\ldots v_n$ be an odd cycle of $H_1$, where $n\geq 4$.
Suppose that $V(C)\subseteq A$.
Since $B$ is non-empty, we choose an arbitrary vertex $b$ in $B$.
Let $b$ be a vertex in $B$.
Since $H_1$ is connected, there exists a path $P$ from $b$ to $V(C)$ internal disjoint from $V(C)$.
Without loss of generality, we assume that $V(P)\cap V(C)=\{v_0\}$.
Then $P\cup C[v_0,v_i]\cup T_{b,v_i}$ for $i=0,1,\ldots, 4$ give $5$ cycles of consecutive lengths, a contradiction.
Therefore, $B\cap V(C)\neq\emptyset$, and similarly, $A\cap V(C)\neq\emptyset$.
Then there must be an $A$-$B$ path of length $4$ in $C$
(otherwise, since $4$ and $|C|$ is co-prime and $|C|\geq 5$, one can deduce that all vertices of $C$ are contained in one of the two parts $A$ and $B$, a contradiction).

Without loss of generality, we may assume that $v_0,v_1\in A$ and $v_2\in B$.
Then $T_{v_1,v_2}\cup v_2v_1$ and $T_{v_0,v_2}\cup v_2v_1v_0$ are two cycles of lengths $2h+1$ and $2h+2$, respectively.
We have showed that there exists some $A$-$B$ path of length $4$ in $C$ which gives a cycle of length $2h+4$,
so we may assume that there is no $A$-$B$ path of length $3$ or $5$ in $C$.
This would force that one of the following holds.

\subsection{There is no $A$-$B$ path of length $3$ in $H_1$.}
This would force that for any path $P'=u_0u_1\ldots u_s$ in $H_1$ with $u_1=v_0,u_2=v_1,u_3=v_2$,
we can derive that $u_j\in B$ if $j\equiv 0$ modulo $3$ and $u_j\in A$ if $j\equiv 1$ or $2$ modulo $3$.
Moreover, we have that $v_{3i}, v_{3i+1}\in A$ and $v_{3i+2}\in B$ for each possible $i\geq 0$.
So $|C|\geq 9$ and $G$ contains a cycle of length $\ell\in\{2h+1,2h+2,2h+4,2h+5,2h+7,2h+8\}$.
In particular, since $H_1$ is connected, for any vertex $b\in B$, there exists a path of length 2 in $H_1$ from $b$ to some vertex in $A$.  And for any bad vertex $a\in A$, there exists a path $b_1aa_1b_2$ satisfying $b_1,b_2\in B$ and $a,a_1\in A$.

\begin{itemize}

\item Suppose that $N_{A\cup A'}(H_2)\neq\emptyset$ and $N_{B\cup B'}(H_2)\neq\emptyset$.
Since $H_2$ is connected and every vertex of $H_2$ has a neighbor in $A\cup A'\cup B\cup B'$,
there exist two adjacent vertices $p,q$ of $H_2$ such that $p$ has a neighbor $p'$ in $A\cup A'$ and $q$ has a neighbor $q'$ in $B\cup B'$.
Then $p'pqq'\cup T_{p',q'}$ is a cycle of length $2h+3$.
It follows that $G$ contains $5$ cycles of lengths $2h+1,2h+2,2h+3,2h+4$ and $2h+5$, respectively, a contradiction.

\item Suppose that $N_{L_i}(H_2)\subseteq B\cup B'$.
Since $N_{A\cup B}(H_2)\neq\emptyset$, we have that $v\in B$.
Let $u$ be any vertex in $N_{H_2}(v)$.
Choose $w_1\in V(H_2)$ such that there exists a path $Q$ of length $2$ from $u$ to $w_1$ in $H_2$.
Let $w_2$ be a neighbor of $w_1$ in $B\cup B'$.
Suppose that $w_2\neq v$.
Note that there is a path $R:=vv''v'$ such that $v''\in B$ and $v'\in A$.
Then $R \cup vu \cup Q \cup w_1w_2\cup T_{w_2,v'}$ is a cycle of length $2h+6$.
So $G$ contains cycles of lengths $2h+4,2h+5,2h+6,2h+7$ and $2h+8$, a contradiction.
Therefore $w_2=v$ and $w_1\in N_{H_2}(v)$. That says, every vertex in $H_2$ of distance $2$ from a neighbor of $v$ is a neighbor of $v$.
Continuing to apply this along with a path from $u$ to an odd cycle $C_0$ in $H_2$,
we could obtain that $v$ is adjacent to all vertices of $C_0$, which contradicts that $G$ is $K_3$-free.
Therefore, $N_{B\cup B'}(H_2)=\emptyset$.

\item Now we see that $N_{L_i}(H_2)\subseteq A\cup A'$.
This forces that $v\in A$.
For any neighbor $u'$ of $v$ in $H_2$,
let $w_3\in V(H_2)$ satisfies that there exists a path $Q'$ of length $2$ from $u'$ to $w_3$ in $H_2$.
Note that $v\in A$ is bad in $H_1$, we can infer that there exists a path $b_2va_1b_1$ in $H_1$ such that $a_1\in A$ and $b_1,b_2\in B$.
Note that $v$ and $a_1$ are symmetric.
Let $w_4$ be a neighbor of $w_3$ in $A\cup A'$.
Suppose that $w_4\notin\{v,a_1\}$.
Then $vu'\cup Q'\cup w_3w_4\cup T_{w_4,b_1}\cup b_1a_1v$ is a cycle of length $2h+6$.
So again, $G$ contains cycles of lengths $2h+4,2h+5,2h+6,2h+7$ and $2h+8$, a contradiction.
Therefore, $w_4\in\{v,a_1\}$. That is, every vertex in $H_2$ of distance $2$ from a neighbor of $v$ or $a_1$ is adjacent to one of $v,a_1$.
Continuing to apply this along with a path from $u'$ to an odd cycle $C_1$ in $H_2$,
we could obtain that every vertex of $C_1$ is adjacent to one of $v,a_1$.
But this would force a $K_3$ in $G$.
This final contradiction completes the proof of this subsection.
\end{itemize}

\subsection{There is an $A$-$B$ path of length $3$ in $H_1$.}
Therefore, we may assume that there is no $A$-$B$ paths of length $5$ in $H_1$.

\begin{cl}
Let $t_1t_2t_3$ be a path in $H_1$ satisfying that $t_1$ and $t_3$ are in different parts.
Then $t_2$ does not have a neighbor in $V(H_2)$.
\end{cl}

\begin{proof}[Proof of Claim]
Without loss of generality, we may assume that $t_1,t_2\in A$ and $t_3\in B$.
Suppose that $t_2$ has a neighbor in $H_2$.
Let $s$ be any vertex in $N_{H_2}(t_2)$.
Choose $s'\in V(H_2)$ such that there exists a path $Q$ of length $2$ from $s$ to $s'$ in $H_2$.
Let $t$ be a neighbor of $s'$ in $L_i-M$.
Suppose that $t\neq t_2$.
If  $t\in A\cup A'$, then $t_3t_2s\cup Q\cup s't\cup T_{t,t_3}$ is a cycle of length $2h+5$.
So $G$ contains cycles of lengths $2h+1,2h+2,2h+3,2h+4$ and $2h+5$, a contradiction.
Therefore $t\in B\cup B'$, then $t_1t_2s\cup Q\cup s't\cup T_{t,t_1}$ is a cycle of length $2h+5$.
So $G$ contains cycles of lengths $2h+1,2h+2,2h+3,2h+4$ and $2h+5$, a contradiction.
Therefore $t=t_2$ and $w_1$ is the neighbor of $t_2$.
That says, every vertex in $H_2$ of distance $2$ from a neighbor of $t_2$ is a neighbor of $t_2$.
Continuing to apply this along with a path from $s$ to an odd cycle $C_2$ in $H_2$,
we could obtain that $t_2$ is adjacent to all vertices of $C_2$, which contradicts that $G$ is $K_3$-free.
This completes the proof of Claim.
\end{proof}

\begin{itemize}
\item Suppose that $N_{A\cup A'}(H_2)\neq\emptyset$ and $N_{B\cup B'}(H_2)\neq\emptyset$.
Suppose that there exists a path $p_0p_1p_2p_3$ in $H_2$ such that $p_0$ is type-$A$ and $p_3$ is type-$B$.
Let $q$ be the neighbor of $p_0$ in $A\cup A'$ and $q'$ be the neighbor of $p_3$ in $B\cup B'$.
Then $qp_0p_1p_2p_3q'\cup T_{q',q}$ is a cycle of length $2h+5$.
So $G$ contains cycles of lengths $2h+1,2h+2,2h+3,2h+4$ and $2h+5$, a contradiction.
This forces that every two vertices which are linked by a path of length $3$ in $H_2$ have the same type.
Note that $N_{A\cup A'}(H_2)\neq\emptyset$ and $N_{B\cup B'}(H_2)\neq\emptyset$.
By symmetry between $A\cup A'$ and $B\cup B'$, there exists a path $z_0z_1z_2$ in $H_2$ such that $z_0$ and $z_1$ are type-$A$ and $z_2$ is type-$B$.
Moreover, for any path $P'':=u_0u_1\ldots u_s$ in $H_2$ with $u_0=z_0,u_1=z_1,u_2=z_2$, we can derive that $u_j$ is type-$A$ if $j\equiv 0$ or $1$ modulo $3$ and $u_j$ is type-$B$ if $j\equiv 2$ modulo $3$.
Moreover, for any path $P''':=u_0u_1\ldots u_s$ in $H_2$ with $u_0=z_2,u_1=z_1,u_2=z_0$, we can derive that $u_j$ is type-$A$ if $j\equiv 2$ modulo $3$ and $u_j$ is type-$B$ if $j\equiv 1$ or $2$ modulo $3$.
This forces that every cycle in $H_2$ has length $0$ modulo $3$.
Since $H_2$ is non-bipartite and $K_3$-free, there is an odd cycle $C_3:=w_0w_1\ldots w_mw_0$ of length at least $9$.
Note that $w_0$ and $w_8$ have different types.
If follows that there is a cycle of length $2h+10$.
So $G$ contains cycles of lengths $2h+1,2h+2,2h+3,2h+4$ and $2h+10$, a contradiction.

\item Therefore, all vertices in $H_2$ have the same type.
Without loss of generality, we may assume that $N_{L_i}(H_2)\subseteq A\cup A'$.
Therefore $v\in A$ and let $f_0$ be a neighbor of $v$ in $H_2$.
Since $H_2$ is $K_3$-free and non-bipartite, there is a path $f_0f_1f_2$ in $H_2$.
Since $H_1$ is a $K_3$-free non-bipartite graph and $v$ is a bad vertex in $H_1$, there is a path $a_0a_1va_2a_3$ in $H_1$.
Since there is no $A$-$B$ path of length $5$ in $H_1$, we have that for any path $Q':=u_0u_1\ldots u_s$ in $H_1$ with $u_0=a_0,u_1=a_1,u_2=v,u_3=a_2,u_4=a_3$, we can derive that $u_j$ and $u_k$ are in the same part if $j\equiv k$ modulo $5$.
Also, we have that for any path $Q':=u_0u_1\ldots u_s$ in $H_1$ with $u_0=a_3,u_1=a_2,u_2=v,u_3=a_1,u_4=a_0$, we can derive that $u_j$ and $u_k$ are in the same part if $j\equiv k$ modulo $5$.
Based on previous analysis, We have that $a_1$ and $a_2$ have the same type.

\begin{itemize}
\item Suppose that $a_1,a_2\in A$.
Since $V(H_1)\cap B\neq\emptyset$, we have that one of $a_0$ and $a_3$ is in $B$.
Without loss of generality, we may assume that $a_0\in B$.
Let $w$ be a neighbor of $f_1$ in $H_1$.
We have that $w\in A\cup A'$.
Since $G$ is $K_3$-free, $w\neq u$.
Note that $a_0a_1v$ satisfying that $a_0$ and $v$ are in different parts of $H_1$.
By Claim, we have that $w\neq a_1$.
Therefore, $wf_1f_0va_1a_0\cup T_{a_0,w}$ is a cycle of length $2h+5$.
So $G$ contains cycles of lengths $2h+1,2h+2,2h+3,2h+4$ and $2h+5$, a contradiction.
\item Therefore, $a_1,a_2\in B$.
Let $w'$ be a neighbor of $v_1$ in $H_1$.
We have that $w'\in A\cup A'$.
Suppose that $w'\neq v$.
Then $w'f_2f_1f_0va_1\cup T_{a_1,w'}$ is a cycle of length $2h+5$.
So $G$ contains cycles of lengths $2h+1,2h+2,2h+3,2h+4$ and $2h+5$, a contradiction.
Therefore $w'=v$.
That says, every vertex in $H_2$ of distance $2$ from a neighbor of $v$ is a neighbor of $v$.
Continuing to apply this along with a path from $f_0$ to an odd cycle $C_4$ in $H_2$,
we could obtain that $v$ is adjacent to all vertices of $C_4$, which contradicts that $G$ is $K_3$-free.
\end{itemize}
\end{itemize}

This completes the proof of Lemma \ref{keyclam}.
\end{proof}

Now, we define a coloring $c:V(G)\rightarrow \{1,2,3,4,5\}$ as following.
Let $D$ be any bipartite component of $G[L_i]$ for some $i$.
If $i$ is even, we color one part of $D$ with color $1$ and the other part with color $2$,
and if $i$ is odd, we color one part of $D$ with color $4$ and the other part with color $5$.
Let $F$ be any non-bipartite component of $G[L_j]$ for some $j$.
If $j$ is even, by using the block structure of $F$, we can properly color $V(F)$ with colors $1,2$ and $3$ by coloring bad vertices with colors $1,2$ and $3$ and coloring good vertices with colors $1$ and $2$.
If $j$ is odd, then we also can properly color $V(F)$ with colors $3, 4$ and $5$ by coloring bad vertices with colors $3,4$ and $5$ and coloring good vertices with colors $4$ and $5$.

Next, we argue that $c$ is a proper coloring on $G$.
Let $H_1$ be a component of $G[L_i]$ and $H_2$ be a component of $G[L_{i+1}]$ for $i\geq0$ such that there exists an edge between $H_1$ and $H_2$.
If one of them is bipartite, then $c$ is proper on $V(H_1)\cup V(H_2)$ .
Therefore, both $H_1$ and $H_2$ are non-bipartite.
By the above claim, all vertices of $H_2$ are not adjacent to vertices of color $3$ in $H_1$.
It follows that $c$ is proper on $V(H_1)\cup V(H_2)$.
Therefore, $c$ is a proper $5$-coloring of $G$, which contradicts that $G$ is $6$-critical.
This completes the proof of Theorem \ref{six}.
\end{proof}

\subsection*{Acknowledgements.}
The author would like to thank Jun Gao and Jie Ma for useful discussions. The author also thanks Jun Gao for carefully reading a draft of this paper.


\begin{thebibliography}{99}

\bibitem{B77}
B. Bollob\'as, \newblock{Cycles modulo k}, \newblock{\emph{Bull. London Math. Soc.}} \textbf{9} (1977), 97--98.

\bibitem{CY}
S. Chiba and T. Yamashita, \newblock{Minimum degree conditions for the existence of cycles of all lengths modulo $k$ in graphs},
\newblock{arXiv:1904.03818} [math.CO] 8 April 2019.

\bibitem{D10}
A. Diwan, \newblock{Cycles of even lengths modulo k}, \newblock{\emph{J. Graph Theory}} \textbf{65} (2010), 246--252.

\bibitem{Erd76}
P. Erd\H{o}s, \newblock{Some recent problems and results in graph theory, combinatorics, and number theory},
\newblock{\emph{Proc. Seventh S-E Conf. Combinatorics, Graph Theory and Computing, Utilitas Math.}}, Winnipeg, 1976, pp. 3--14.

\bibitem{Fan02}
G. Fan,
\newblock{Distribution of cycle lengths in graphs},
\newblock{\emph{J. Combin. Theory Ser. B}} \textbf{84} (2002), 187--202.

\bibitem{GHLM}
J. Gao, Q. Huo, C. Liu and J. Ma,
\newblock{A unified proof of conjectures on cycle lengths in graphs},
\newblock{\emph{Int. Math. Res. Not.}}, to appear.

\bibitem{GHM20}
J. Gao, Q. Huo and J. Ma,
\newblock{A strengthening on odd cycles in graphs of given chromatic number},
\newblock{arXiv:2012.10624} [math.CO] 19 December 2020.

\bibitem{LM}
C. Liu and J. Ma,
\newblock{Cycle lengths and minimum degree of graphs},
\newblock{\emph{J. Combin. Theory Ser. B}} \textbf{128} (2018), 66--95.

\bibitem{MW}
B. Moore and D. B. West,
\newblock{Cycles in color-critical graphs},
\newblock{arXiv:1912.03754v2} [math.CO].

\bibitem{SV17}
B. Sudakov and J. Verstra\"ete,
\newblock{The extremal function for cycles of length $l$ mod $k$},
\newblock{\emph{Elec. J. of Combin.}} \textbf{24(1)} (2017), \#P1.7

\bibitem{Th83}
C. Thomassen, \newblock{Graph decomposition with applications to subdivisions and
path systems modulo k}, \newblock{\emph{J. Graph Theory}} \textbf{7} (1983), 261--271.

\bibitem{Th88}
C. Thomassen, \newblock{Paths, circuits and subdivisions},
\newblock{\emph{Selected Topics in Graph Theory (L. Beineke and R. Wilson, eds.)}}, vol. 3, Academic Press, 1988, pp. 97--131.

\bibitem{V00}
J. Verstra\"ete,
\newblock{On arithmetic progressions of cycle lengths in graphs},
\newblock{\emph{Combin. Probab. Comput.}} \textbf{9} (2000), 369--373.

\end{thebibliography}
\end{document}